\documentclass[12pt]{article}

\usepackage{graphicx, amsfonts,amssymb, 
	amsmath,amsthm,color,mathtools,tikz,tikz-3dplot,hyperref,url}

\usepackage{epstopdf}
\usepackage{wasysym}
\usetikzlibrary{calc,intersections,through,backgrounds, arrows,decorations.markings,arrows.meta}
\usepackage{tkz-euclide}

\newtheorem{lemma}{Lemma}[section]

\newtheorem{conjecture}{Conjecture}
\newtheorem{theorem}{Theorem}

\theoremstyle{definition}
\newtheorem{remark}[lemma]{Remark}

\graphicspath{{fig/}}

\begin{document}

\def\mainfile{}
\newcommand{\eps}{{\varepsilon}}
\newcommand{\C}{{\mathbb C}}
\newcommand{\Q}{{\mathbb Q}}
\newcommand{\R}{{\mathbb R}}
\newcommand{\Z}{{\mathbb Z}}
\newcommand{\RP}{{\mathbb {RP}}}
\newcommand{\CP}{{\mathbb {CP}}}
\newcommand{\Tr}{\rm Tr}
\newcommand{\g}{\gamma}
\newcommand{\G}{\Gamma}
\newcommand{\e}{\varepsilon}

\title{Remarks on rigidity properties of conics}

\author{
Serge Tabachnikov\footnote{
Department of Mathematics,
Pennsylvania State University,
University Park, PA 16802;
tabachni@math.psu.edu}
}

\date{\today}
\maketitle

\begin{abstract} Inspired by the recent results toward Birkhoff conjecture (a  rigidity property of billiards in ellipses), 
we discuss two rigidity properties of conics. The first one concerns  symmetries of an analog of polar duality associated with an oval, and the second  concerns  properties of the circle map associated with an oval and two pencils of lines.
\end{abstract}

\section{Polar duality} \label{sect:dual}

Given a non-degenerate conic in the projective plane, polar duality associates a line to a point and a point to a line. 
Polar duality preserves the incidence relation between points and lines. See Figure \ref{dual1}. 

\begin{figure}[ht]
\centering
\includegraphics[width=.45\textwidth]{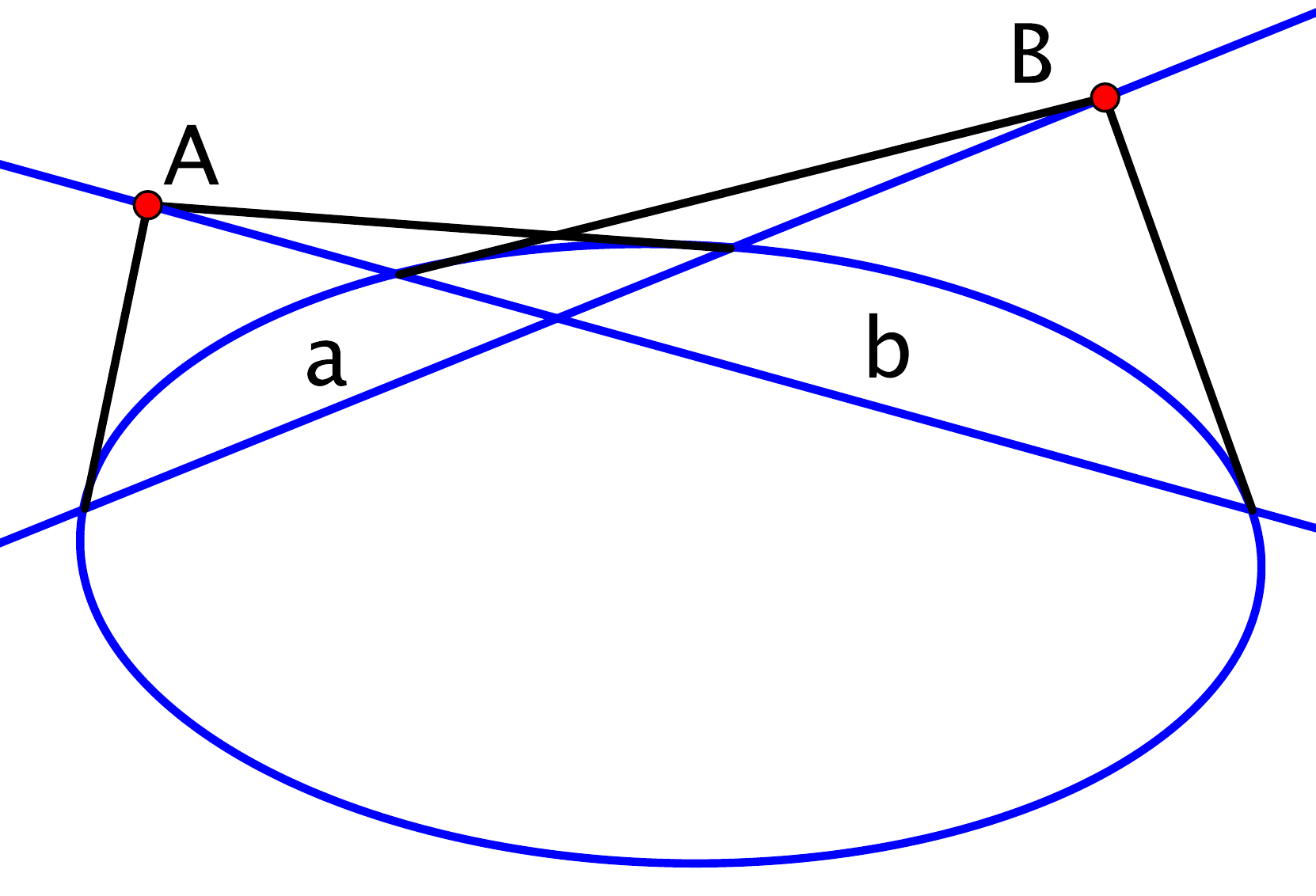}
\caption{Line $a$ is polar dual to point $A$, and line $b$ to point $B$: if $B$ lies on $a$, then $b$ contains $A$.}	
\label{dual1}
\end{figure}

Let $\g$ be a germ of a smooth convex curve in the plane. If point $A$ lies on the concave side of $\g$ and is sufficiently close to $\g$, then one can draw two tangent lines to $\g$, connect the tangency points, and declare the resulting line $a$ ``dual" to point $A$. When does this ``duality" preserve the incidence relation between points and lines?

Let $B$ be a point on line $a$ that lies on the concave side of $\g$ and is sufficiently close to $\g$, so that this construction can be repeated, yielding the line $b$. 

\begin{theorem} \label{thm:symm}
Assume that for all admissible choices of points $A$ and $B$ as described above, the line $b$ passes through point $A$. Then $\g$ is an arc of a conic.
\end{theorem}

\begin{proof} The proof is computational, with the heavy lifting done by Mathematica.

We give the curve affine parameterization, $\g(t)$, such that $\det(\g',\g'')=1$ for all $t$. This implies that 
\begin{equation} \label{curv}
\g'''(t)=-k(t)\g'(t),
\end{equation}
 where the function $k$ is called the affine curvature. Affine curvature is constant if and only if the curve is a conic. See, e.g., \cite{Si}.

Scale and reparameterize the curve: $\G(t)=r\g(ct)$. If $c^2r^3=1$, then $\det(\G',\G'')=1$, and $\G'''=r^2k\G'$. Thus, by scaling and reparameterizing the curve, one can scale the affine curvature by a constant positive factor. 

Fix $t$ and consider a neighborhood of point $\g(t)$. We assume that $k(t)\neq 0$: indeed, if the affine curvature is identically zero, the curve is a parabola, and we are done. Using the above described rescaling, we may assume that $k(t)=\pm1$. In what follows we consider the case $k(t)=1$; the other case is  analogous. 

Chose the coordinates so that 
$$
\g(t)=(0,0), \g'(t)=(1,0), \g''(t)=(0,1).
$$
 Then $\g'''(t)=(-1,0)$. Differentiating (\ref{curv}) once and twice and using $k(t)=1$, we find that 
 $$
 \g^{\rm iv}(t)=-(p,1), \g^{\rm v}(t)=(1-q,-2p),
 $$
  where $p=k'(t), q=k''(t)$. 

We want to show that $p=0$. Since $t$ is arbitrary, this would imply that $\g$ has constant affine curvature, and hence it is a conic. 

Fix $\e$ and consider three points of the curve $\g(t-\e),\g(t+a\e),\g(t+\e)$. Construct the fourth point $\g(t-b\e)$ as shown in Figure \ref{curve}. 
\begin{figure}[ht]
\centering
\includegraphics[width=.75\textwidth]{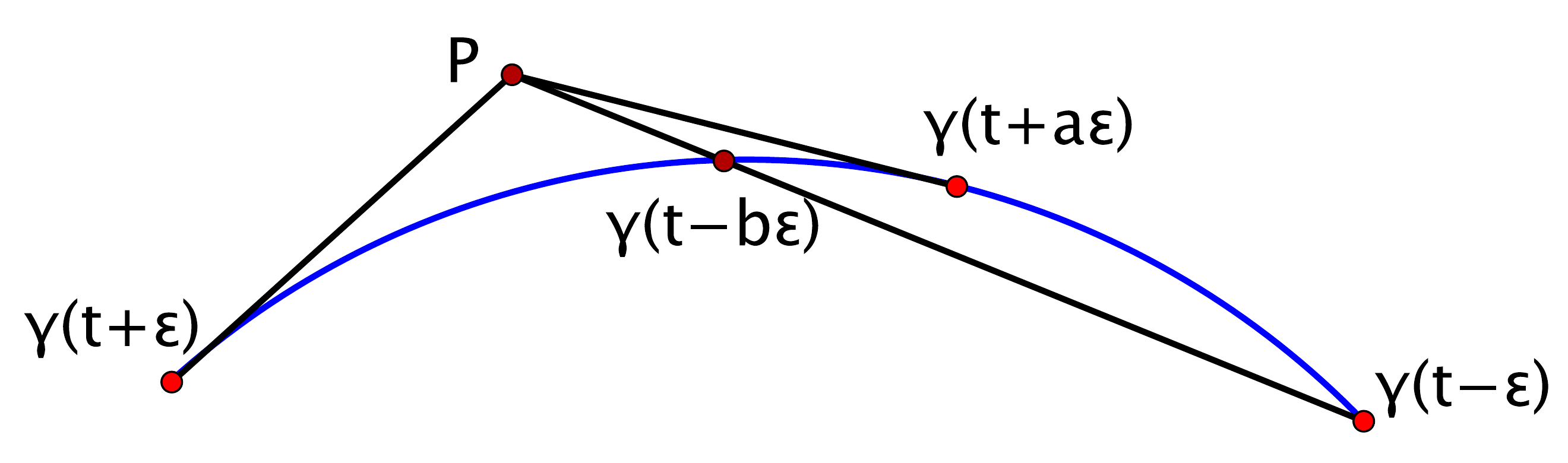}
\caption{Constructing $b(a,\e)$.}	
\label{curve}
\end{figure}

This defines the function $b(a,\e)$.  Our assumption that the ``duality" preserves the incidence relation implies that the map $f:a \mapsto  b(a,\e)$ is an involution for every $a\in (-1,1)$.
We shall calculate the Taylor expansion of $b(a,\e)$ in $\e$ up to $\e^3$ and show that if $f$ is an involution, then $p=0$.

It is convenient to lift points from the plane to $\R^3$ by adding 1 as the third coordinate, that is, to consider $\R^2$ as an affine plane in $\R^3$ at height 1. In $\R^3$ one has the duality between 1- and 2-dimensional subspaces given by the dot product. Thus the plane spanned by vectors $u$ and $v$ is identified with its normal $u\times v$.

We use cross-product to find intersection of two lines in the plane (that is, two 2-dimensional subspaces in 3-space) and the line through a pair of points (that is, the plane spanned by two 1-dimensional subspaces in 3-space). 

For example, let $u=(u_1,u_2)$ and $v=(v_1,v_2)$ be two vectors in $\R^2$. Then the line in $\R^2$ through $u$ in the direction $v$ becomes the plane in $\R^3$ spanned by the vectors $(u_1,u_2,1)$ and $(u_1+v_1,u_2+v_2,1)$, which is identified with its normal vector $(u_1,u_2,1)\times (v_1,v_2,0)$.

With these identifications, one has
$$
P= [\g(t+a\e)\times \g'(t+a\e)]\times [\g(t+\e)\times \g'(t+\e)]
$$
and
\begin{equation} \label{eq:cond1}
\det[P,\g(t+b\e),\g(t-\e)]=0.
\end{equation}
Using the properties of cross-product
$$
\det(u,v,w)=(u\times v)\cdot w, (u\times v)\cdot (w\times x)=(u\cdot w)(v\cdot x)-(u\cdot x)(v\cdot w),
$$
we rewrite (\ref{eq:cond1}) as
\begin{equation} \label{eq:cond2}
\begin{split}
&\det[\g(t+a\e),\g'(t+a\e),\g(t+b\e)]\det[\g(t+\e),\g'(t+\e),\g(t-\e)]\\
&=\det[\g(t+a\e),\g'(t+a\e),\g(t-\e)]\det[\g(t+\e),\g'(t+\e),\g(t+b\e)].
\end{split}
\end{equation}

To solve this equation for $b(a,\e)$, we write the terms involved as power series in $\e$:
\begin{equation} \label{eq:Tay}
\begin{split}
&\g(t+a\e)=\\
&\left(a\e-\frac{a^3\e^3}{6}- \frac{a^4\e^4}{24}p+\frac{a^5\e^5}{120}(1-q)+O(\e^6), \frac{a^2\e^2}{2} -\frac{a^4\e^4}{24} -\frac{a^5\e^5}{60}p +O(\e^6),1\right),\\
&\g'(t+a\e)=\\
&\left(1-\frac{a^2\e^2}{2}- \frac{a^3\e^3}{6}p+\frac{a^4\e^4}{24}(1-q)+O(\e^5), a\e -\frac{a^3\e^3}{6} -\frac{a^4\e^4}{12}p +O(\e^5),0\right),
\end{split}
\end{equation}
and likewise for $\g(t\pm\e), \g(t+\e)$, and  $\g(t-b\e)$, where 
$$
b=b_0(a)+\e b_1(a)+\e^2b_2(a)+\e^3b_3(a)+O(\e^4).
$$
 We want to find the functions $b_i(a), i=0,1,2,3$.

Substitute (\ref{eq:Tay}) in (\ref{eq:cond2}) to obtain an equation involving the terms of degrees 4 to 7 in $\e$,  an equation mod $\e^8$. Solving this equation in Mathematica, one finds
\begin{equation*}
\begin{split}
b(a,\e)=-\frac{3a+1}{a+3}
&+\frac{2(a-1)^2(a+1)^2}{3(a+3)^3} \e^2 \\
&+\frac{2(a-1)^2(a+1)^2(2a^2+9a+5)}{15(a+3)^4} p \e^3 
 +O(\e^4).
\end{split}
\end{equation*}
It follows that 
\begin{equation*}
\begin{split}
f(f(a))=a - \frac{(a-1)^2(a+1)^2(a^2+6a+1)}{24(a+3)^2} p \e^3+O(\e^4).
\end{split}
\end{equation*}
Since $f$ is an involution, the cubic in $\e$ term vanishes, and since  $a$ is arbitrary, we conclude that $p=0$, as needed.
\end{proof}

\section{A family of circle maps} \label{sect:maps}

Let $\g$ be an oval. The intersections of $\g$ with a family of parallel lines having direction $u$
define an involution $f_u:\g\to\g$. Let $u$ and $v$ be two directions, and let $F_{u,v}=f_u\circ f_v$. Then $F_{u,v}$ is an orientation-preserving circle homeomorphism. 

This circle map was studied from a surprising variety of perspectives \cite{Ar,GKT,HM,Jo,KT,NT,So}: the Dirichlet problem for a 2nd order hyperbolic PDE, Hilbert's 13th problem, billiards in pseudo-Euclidean geometry, ``chess billiards", etc.

If $\g$ is a circle, then $F_{u,v}$ is a rotation. Since the map $F_{u,v}$ commutes with affine transformations, $F_{u,v}$ is conjugated to a rotation if $\g$ is an ellipse. We repeat here a conjecture made in \cite{GKT}.

\begin{conjecture} \label{conj1}
Suppose that $F_{u,v}:\g\to\g$ is conjugated to a rotation for all pairs of directions $u,v$. Then $\g$  is an ellipse.
\end{conjecture}

As far as we know, this conjecture is open. Let us consider its projective version, where a family of parallel lines is replaced by a pencil of lines passing through a fixed point.

Let $P$  be a point in the exterior of $\g$. The intersections of $\g$ with the lines through $P$ define an involution $f_P:\g\to\g$ (if $P$ is a point at infinity, we are back to the previous situation). For two points, $P$ and $Q$, outside of $\g$, let $F_{P,Q}=f_P\circ f_Q$. Assume that the line $PQ$ is disjoint from $\g$.

If $\g$ is an ellipse, then $F_{P,Q}$ is conjugated to a rotation. Indeed, a projective transformation that takes the line $PQ$ to the line at infinity takes $\g$ to an ellipse (because $PQ$ does not intersect $\g$) and conjugates  $F_{P,Q}$ with $F_{u,v}$, which is the case discussed above.

\begin{theorem} \label{thm:PQ}
If $F_{P,Q}:\g\to\g$ is conjugated to a rotation for all pairs of points $P,Q$ such that $PQ\cap \g=\emptyset$, then $\g$ is an ellipse.
\end{theorem}

\begin{proof}
Consider the left part of Figure \ref{invol}. One has 
$$
f_P(x)=x,\ f_P(y)=y,\ f_Q(x)=y,\ f_Q(y)=x,
$$
hence $F_{P,Q} (x)=y, F_{P,Q}(y)=x$. Since $F_{P,Q}$ is conjugated to a rotation, $F_{P,Q}$ is an involution:
$$
f_P\circ f_Q\circ f_P\circ f_Q=Id\ \ {\rm or}\ \ f_Q\circ f_P\circ f_Q=f_P.
$$
This implies that points $P,w,z$ on the right part of Figure \ref{invol} are collinear.

\begin{figure}[ht]
\centering
\includegraphics[width=.45\textwidth]{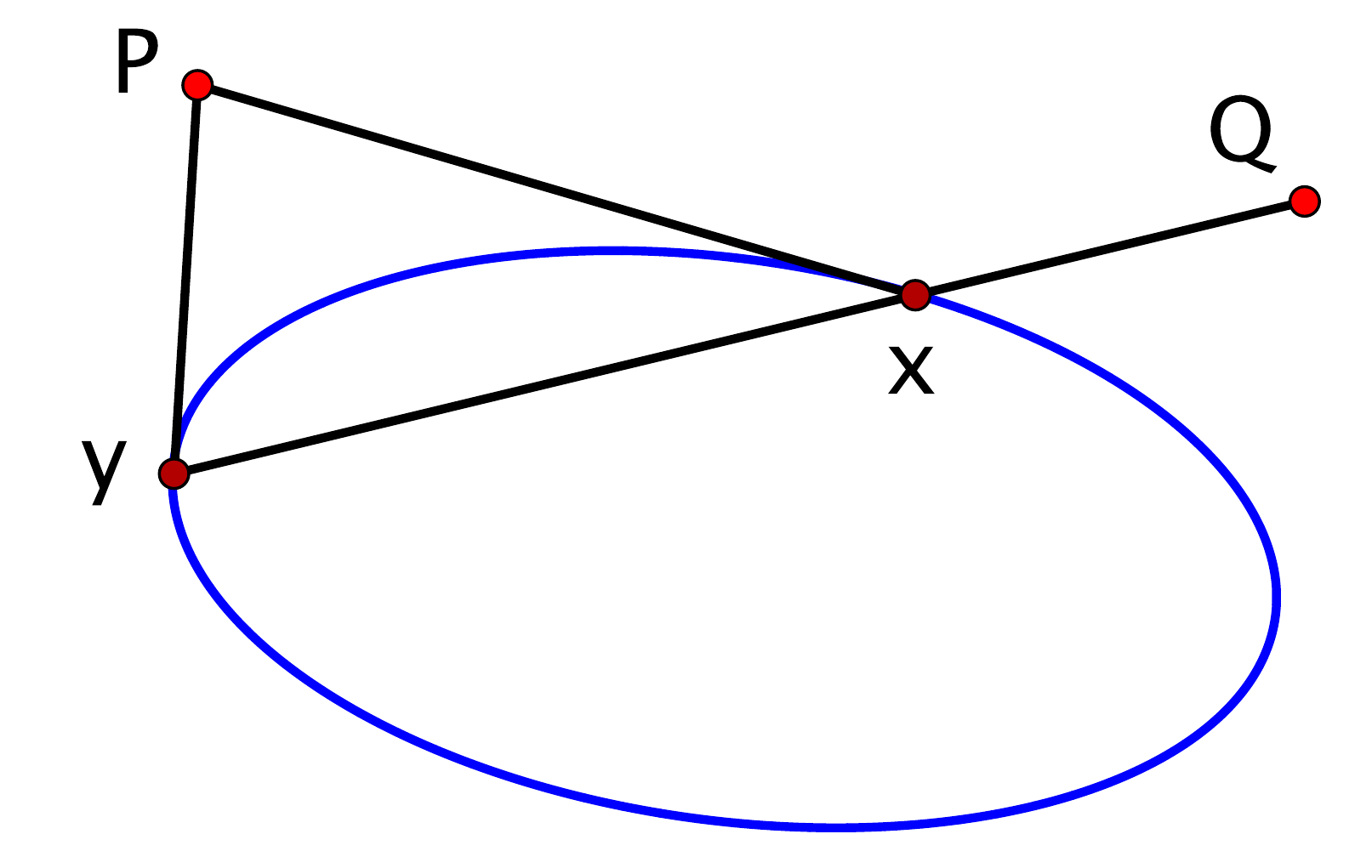}\qquad
\includegraphics[width=.45\textwidth]{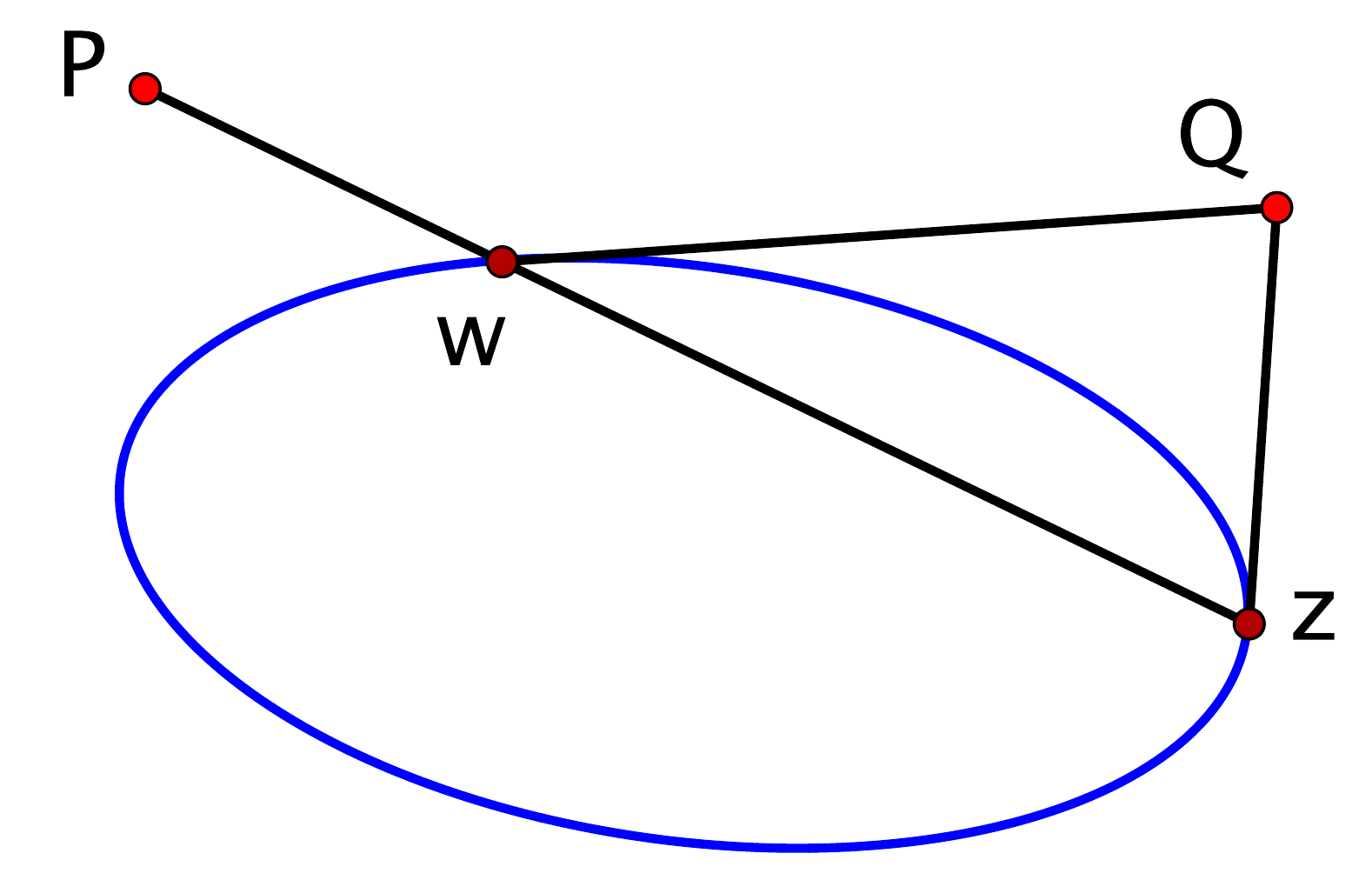}
\caption{Establishing the the incidence preservation property of the polarity associated with an oval.}	
\label{invol}
\end{figure}

It follows that the curve $\g$ has the property that the associated ``polar duality" preserves the incidence relation, and Theorem \ref{thm:symm} implies  that $\g$ is an ellipse.
\end{proof}

If $\g$ is an ellipse, then its interior can be considered as the hyperbolic plane in the projective,  Cayley-Klein, model, and $\g$ is the circle at infinity. Then the map $f_P$ is the restriction to $\g$ of the reflection of the hyperbolic plane in the line polar dual to $P$, and likewise for $f_Q$. If the line $PQ$ is disjoint from $\g$, then the two axes of reflection intersect in $H^2$, and $F_{P,Q}$ is a rotation about this intersection point (through the angle twice that between the axes).

But if the line $PQ$ intersects the ellipse, then the axes of reflection are disjoint in $H^2$, and the
two  points $PQ\cap\g$ are the fixed points of the map $F_{P,Q}$. 
The circle at infinity is identified with $\RP^1$, and the restriction of an orientation preserving isometry to it is a M\"obius (projective) transformation that, in this case, has two fixed points. In the limiting case of a tangent line, $F_{P,Q}$ has a unique fixed point, the tangency point.

Likewise, if points $P$ and $Q$ lie in the interior of $\g$, then the composition of the reflections of $H^2$ in these points is  a hyperbolic orientation preserving isometry of the hyperbolic plane. Its restriction to the circle at infinity again is a M\"obius transfrmation $F_{P,Q}$, and it has two fixed points, the intersection points of the line $PQ$ with $\g$. 
This prompts another conjecture.

\begin{conjecture} \label{conj:int}
Let points $P$ and $Q$ both lie in the exterior of the oval $g$ so that the line $PQ$ intersects it, or both lie in the interior of $\g$. Assume that for 
all such pairs of points the map $F_{P,Q}:\g\to\g$ is $C^\infty$-conjugated to a M\"obius transformation. Then $\g$ is an ellipse.
\end{conjecture}

In the spirit of Conjecture \ref{conj1}, perhaps it is enough for the same conclusion to assume that the points $P$ and $Q$ are confined to one line that intersects the curve $\g$.

\begin{remark}
{\rm The derivatives of a M\"obius transformation at its two fixed points are reciprocal. Being a composition of two inversions, the map $F_{P,Q}$ automatically shares this property. 

A diffeomorphism of $\RP^1$ that has exactly two fixed points is topologically conjugated to a M\"obius transformation, but there exists a functional obstruction to a smooth conjugacy. 

Assume, without loss of generality, that the fixed points of a diffeomorphism $F$ are zero and infinity, and the derivatives at these points are equal and distinct from to $1$. Then $F$ can be uniquely (up to a constant factor)  linearized in $\RP^1-\{0\}$ and in $\RP^1-\{\infty\}$, and the transition map between these two linear coordinates is an obstruction. 

I am grateful to M. Lyubich for explaining this construction to me.
}
\end{remark}


As to  Conjecture \ref{conj1},   it holds under the additional assumption that the curve is centrally symmetric. The idea of the proof is contained in \cite{GKT}; we reproduce it here for completeness.

\begin{theorem} \label{thm:symm}
Let $\g$ be a centrally symmetric oval. If $F_{u,v}:\g\to\g$ is conjugated to a rotation for all pairs of directions $u,v$, then $\g$  is an ellipse.
\end{theorem}

\begin{proof}
We start in the same way as in the proof of Theorem \ref{thm:PQ}, but with the points $P$ and $Q$ lying at infinity, so that the two respective pencils of lines comprise two parallel families. 

Namely, let $u$ be a direction. Consider the two support lines to $\g$ parallel to $u$, and let $v$ be the direction of the segment that connects  the tangency points of these lines with $\g$ (this segment is the {\it affine diameter} of $\g$ in direction $v$). The argument from the proof of Theorem \ref{thm:PQ} shows that the map $F_{u,v}$ is an involution. 

The directions $u$ and $v$ are called conjugate. The fact that $F_{u,v}$ is an involution implies that  this conjugacy relation is symmetric.

It follows that $\g$ admits a family of inscribed parallelograms whose sides have directions $u$ and $v$, and these parallelograms interpolate between the affine diameters in directions $u$ and $v$, see the left part of see Figure \ref{coin}.

Let the center of symmetry of $\g$ be the origin. We claim that the parallelograms are origin-centered. Indeed, given an inscribed parallelogram, the parallelogram that is origin symmetric to it  is also inscribed in $\g$. The strict convexity of $\g$ implies that these two parallelograms coincide, see the right part of Figure \ref{coin}. Hence the parallelograms are centered at the origin.

\begin{figure}[ht]
\centering
\includegraphics[width=.4\textwidth]{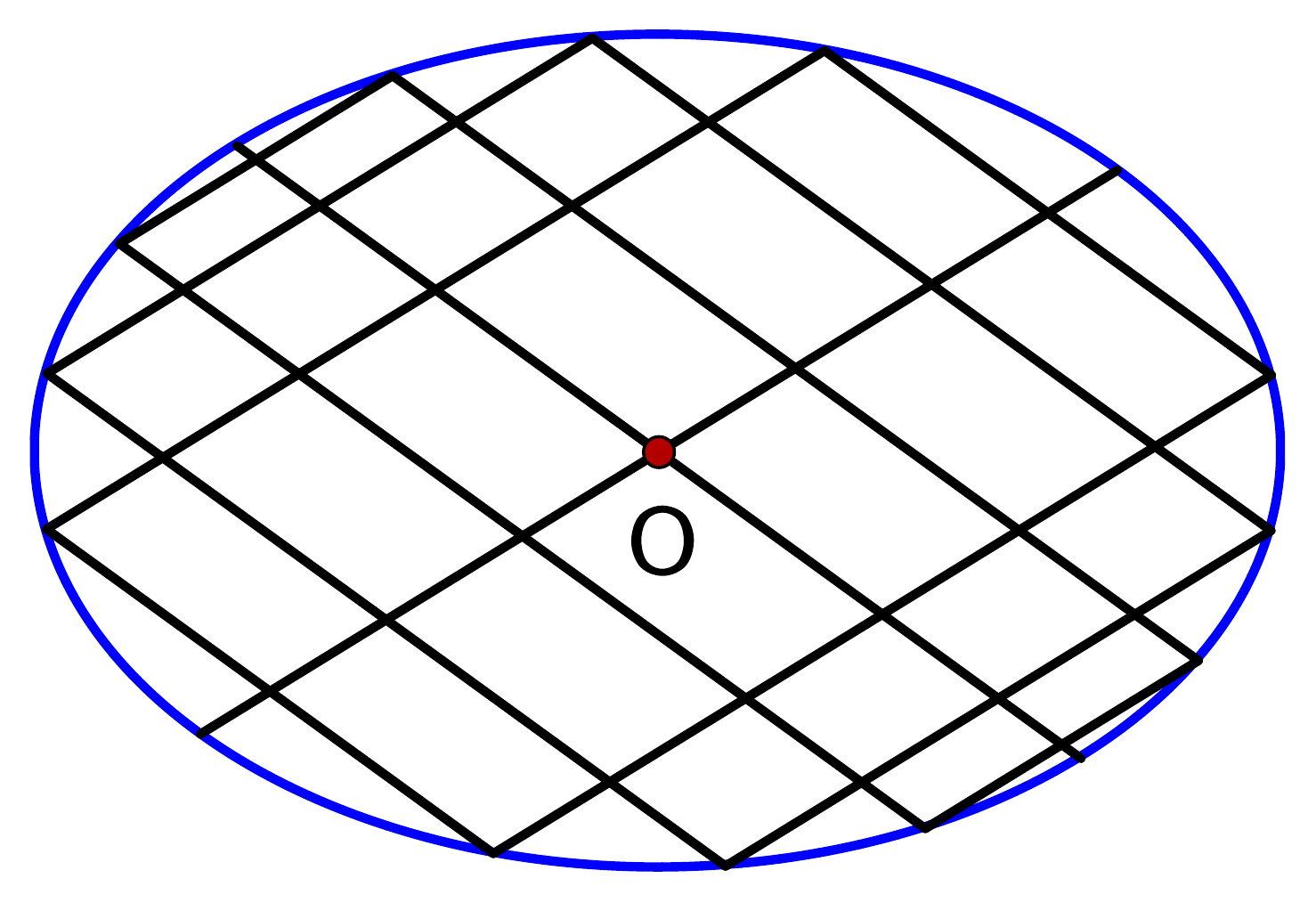} \quad\quad
\includegraphics[width=.3\textwidth]{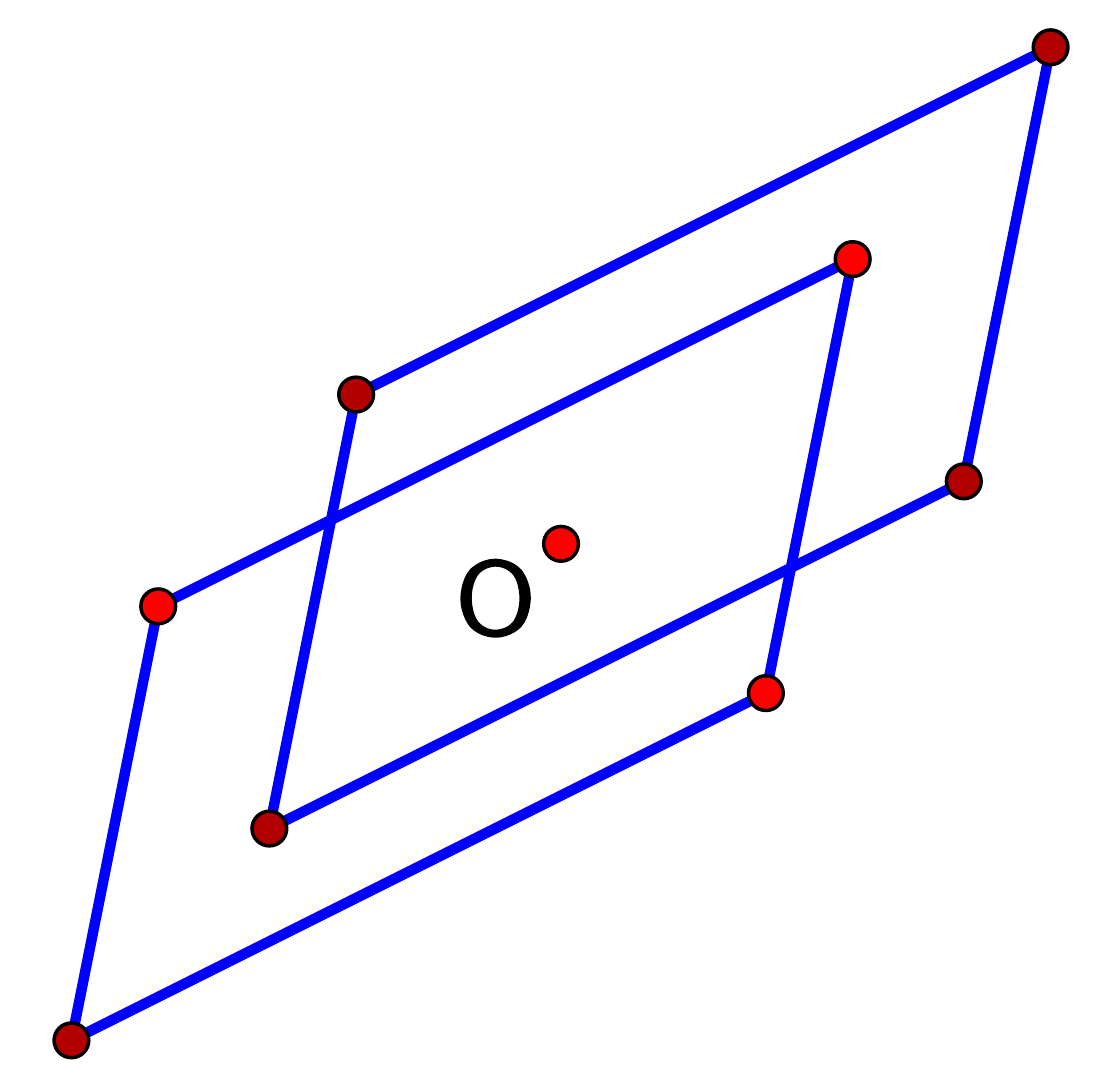}
\caption{Left: a family of inscribed parallelograms whose sides have conjugate directions. Right:
if two centrally symmetric parallelograms do not coincide, their vertices cannot all lie on a strictly convex curve.}	
\label{coin}
\end{figure}

It follows that the midpoints of the sides of the parallelograms lie on the affine diameter having the conjugate direction to the direction of the sides. Since $u$ was arbitrary, $\g$ admits an affine line of symmetry in every direction. This property is characteristic of ellipses -- see \cite{Be} for a proof using the John-Loewner ellipse -- and this completes the proof of the theorem.
\end{proof}
\medskip

{\bf Acknowledgements}. Many thanks to 
M. Lyubich for useful discussions.
The author was supported by NSF grant DMS-2005444.

\end{document}